\documentclass [reqno]{amsart}

\usepackage{amsmath,amsfonts,amssymb,amscd,verbatim,delarray,verbatim}

\theoremstyle{plain}

\def\Q{{\mathbb Q}}
\def\Z{{\mathbb Z}}
\def\C{{\mathbb C}}

            \def\mult{\mathrm{mult}}
             \def\Mat{\mathrm{Mat}}

\def\tr{\mathrm{tr}}

\def\End{\mathrm{End}}

\def\Hom{\mathrm{Hom}}

\def\Mat{\mathrm{Mat}}

                        \def\II{\mathrm{Id}}

\def\fchar{\mathrm{char}}

\def\GL{\mathrm{GL}}
\def\sL{\mathrm{SL}}

                              \def\sL{\mathrm{sl}}

\def\dim{\mathrm{dim}}

\newtheorem{thm}{Theorem}[section]
\newtheorem{lem}[thm]{Lemma}

\theoremstyle{definition}
\newtheorem{defn}[thm]{Definition}

\newtheorem{exs}[thm]{Examples}

           \newtheorem{rem}[thm]{Remark}

\title[Matrices of endomorphisms] {On matrices of endomorphisms of abelian varieties}

\author{Yuri  G. Zarhin}
\address{Pennsylvania State University, Department of Mathematics, University Park, PA 16802, USA}
\email{zarhin@math.psu.edu}
\thanks{The author  was partially supported by Simons Foundation Collaboration grant   \# 585711.
Part of this work was done in December 2019 - January 2020 during my stay at the Weizmann Institute of Science Department of Mathematics,
whose hospitality and support are gratefully acknowledged.}

\begin{document}\date{}

\begin{abstract}

We study endomorphisms of abelian varieties and their action on the $\ell$-adic Tate modules.
We prove that for every endomorphism one may choose a basis of each $\Q_{\ell}$-Tate module
such that the corresponding matrix has rational entries and does not depend on $\ell$.

2010 Math. Subj. Class: Primary 14K05; Secondary 16K20

Key words and phrases: {\sl Abelian varieties, Tate modules, Semisimple algebras}
\end{abstract}

\maketitle

 \section{Introduction}
 
 Let $X$ be an  abelian variety of positive dimension $g$ over an algebraically closed field $K$ of arbitrary characteristic.  We write 
 $\End(X)$ for 
 the endomorphism ring of $X$ and $\End^0(X)$ for the corresponding $\Q$-algebra
 $$\End^0(X):=\End(X)\otimes\Q,$$
 which is a finite-dimensional semisimple algebra over $\Q$. If $n$ is any integer then we write $n_X \in \End(X)$ for multiplication by $n$ in $X$,
 which is an isogeny if $n \ne 0$.
 For example, $1_X$ is the identity automorphism of $X$.

 One may view
 $\End(X)=\End(X)\otimes 1$ as an {\sl order} in $\End^0(X)$.
 Let $\ell \ne \fchar(K)$ be a prime and 
 $T_{\ell}(X)$ be the $\ell$-adic Tate module of $X$, which is a free $\Z_{\ell}$-module of rank $2g$ \cite{Mumford}. We write $V_{\ell}(X)$ for the corresponding
 $\Q_{\ell}$-vector space
 $$V_{\ell}(X)=T_{\ell}(X)\otimes_{\Z_{\ell}}\Q_{\ell}$$
 of dimension $2g$.
 
 By functoriality, there is the natural injective ring homomorphism \cite{Mumford}
 $$\End(X)\to \End_{\Z_{\ell}}(T_{\ell}(X)),  u \mapsto u_{\ell}$$
 that extends by $\Z_{\ell}$-linearity to the injective $\Z_{\ell}$-algebra homomorhism
 $$\End(X)\otimes\Z_{\ell} \hookrightarrow \End_{\Z_{\ell}}(T_{\ell}(X)),  \ u \mapsto u_{\ell}$$
 and extends by $\Q_{\ell}$-linearity to the injective $\Q_{\ell}$-algebra homomorhism
 $$\End^0(X)\otimes_{\Q}\Q_{\ell}=\End(X)\otimes\Q_{\ell} \hookrightarrow \End_{\Q_{\ell}}(V_{\ell}(X)),  \ u \mapsto u_{\ell}$$
 (see \cite{Tate,Mumford}). 
 
 If $u \in \End^0(X)$ then let us consider the  monic degree $2g$ characteristic polynomial
 $$\mathcal{P}_u(t):=\det(t \mathrm{Id}-u_{\ell}, V_{\ell}(X)) \in \Q_{\ell}[t]$$
 of  $u_{\ell}\in \End_{\Q_{\ell}}(V_{\ell}(X))$. (Here $ \mathrm{Id}: V_{\ell}(X) \to V_{\ell}(X)$ is the identity map.)
 A classical result of Weil \cite[Sect. 19,  Th. 4 on p. 180 and Definition on p. 182]{Mumford} asserts that 
 if $u \in \End(X)$ then $\mathcal{P}_u(t)$ lies in $\Z[t]$ and does {\sl not} depend on the choice of $\ell$.
 It follows readily that if $u \in \End^0(X)$ then $\mathcal{P}_u(t)$ lies in $\Q[t]$ and does {\sl not} depend on the choice of $\ell$.
 
 The aim of this note is to prove the following assertion.

\begin{thm}
\label{mat}
Let $X$ be an  abelian variety of positive dimension $g$ over an algebraically closed field $K$ of arbitrary characteristic.   Let $u\in \End^0(X)$.  Then there exists a square matrix $M(u)$ of size $2g$ with entries in $\Q$ that enjoys the following property.

If $\ell$ is any  prime  $\ne \fchar(K)$ then there is a basis of the $2g$-dimensional $\Q_{\ell}$-vector space $V_{\ell}(X)$ 
 such that the corresponding matrix of $u_{\ell}\in \End_{\Q_{\ell}}(V_{\ell}(X))$ coincides with $M(u)$.
\end{thm}

\begin{rem}
In the case of characteristic 0 the assertion of Theorem \ref{mat} reduces to the case of $K=\C$ where it is obvious; in addition, 
if $u\in \End(X)$ then one may chose bases in the free $\Z_{\ell}$-modules $T_{\ell}(X)$ in such a way that the corresponding matrices
$M(u)$ have entries in $\Z$ and do {\sl not} depend on the choice of $\ell$.
\end{rem}

\begin{exs}
\label{isotype}
\begin{itemize}
\item[(i)]
Let $E$ be a subfield of $\End^0(X)$ that contains $1_X$. Then $E$ is a number field; let $d$ be its degree over $\Q$. The natural $\Q_{\ell}$-algebra homomorphism
\begin{equation}
\label{ElEndX}
E\otimes_{\Q}\Q_{\ell} \to \End_{\Q_{\ell}}(V_{\ell}(X)), \ u \mapsto u_{\ell}
\end{equation}
 is injective \cite{Tate,Mumford} and
endows $V_{\ell}(X)$ with the structure of a faithful $E\otimes_{\Q}\Q_{\ell}$-module. It is known \cite{Ribet} that this module is {\sl free}
and its rank is $$h:=\frac{2g}{d}$$
(in particular,  $d\mid 2g$). Let us choose an $h$-element basis  
$$\{z_1, \dots, z_h\}\subset V_{\ell}(X)$$
 of the $E\otimes_{\Q}\Q_{\ell}$-module $V_{\ell}(X)$. Choose also a $d$-element basis
 \begin{equation}
 \label{BE}
 B_E=\{\alpha_1, \dots, \alpha_d\}\subset E
 \end{equation} 
 of the
$\Q$-vector space $E$.  Then the $d$-element set
\begin{equation}
\label{BEl}
B_{E,\ell}:=\{\alpha_1=\alpha_1\otimes 1, \dots, \alpha_d=\alpha_d\otimes 1\}\subset E\otimes_{\Q}\Q_{\ell}
\end{equation}
is a basis of the $d$-dimensional $\Q_{\ell}$-vector space 
$E\otimes_{\Q}\Q_{\ell}$. It follows that the $2g (=dh)$-element 
  set 
  \begin{equation}
  \label{BlX}
  B_{E,\ell,X}:=\{\alpha_{i,\ell} z_j\mid 1\le i\le d, \ 1 \le j \le h\}\subset V_{\ell}(X)
  \end{equation}
   is a basis of the $\Q_{\ell}$-vector space $V_{\ell}(X)$. 
  (Here $\alpha_{i,\ell}\in \End_{\Q_{\ell}}( V_{\ell}(X))$  is the image of 
  $$\alpha_i=\alpha_i\otimes 1\in E\otimes_{\Q}\Q_{\ell}$$
  under the injective homomorphism \eqref{ElEndX}.)
  Now assume that
 $$u \in E=E\otimes 1\subset E\otimes_{\Q}\Q_{\ell},$$
 and let $M_0(u)$ be the matrix of the $\Q$-linear map
 $$\mult_u: E \to E, \  w \mapsto uw \ \forall w \in E$$
 with respect to $B_E$ \eqref{BE}. Clearly, $M_0(u)$ coincides with the matrix of the $\Q_{\ell}$-linear map
$$ E\otimes_{\Q}\Q_{\ell}\to E\otimes_{\Q}\Q_{\ell}, \ w \mapsto uw$$
with respect to $B_{E,\ell}$ \eqref{BEl}.
It follows that  the matrix $M(u)$ of $u_{\ell}$ w.r.t. $B_{E,\ell,X}$ \eqref{BlX} is the block diagonal matrix,
all whose diagonal entries coincide with $M_0(u)$. In particular, all the entries  of  $M(u)$ are rational numbers and $M(u)$ does not depend on the choice of $\ell$.
Notice that we use the same basis $B_{E,\ell,X}$   for all $u \in E$.
\item[(ii)]
Let $m$ be a positive integer and let us consider the abelian variety $Y=X^m$. Then
$\End^0(Y)=\Mat_m(\End^0(X))$ and there is the natural embedding
$$\Mat_m(E)\subset \Mat_m(\End^0(X))=\End^0(Y).$$
We also have
$$V_{\ell}(Y)=\oplus_{i=1}^m V_{\ell}(X),  \ \End_{\Q_{\ell}}(V_{\ell}(Y))=\Mat_m(\End_{\Q_{\ell}}(V_{\ell}(X))).$$
The basis $B_{E,\ell,X}$ of $V_{\ell}(X)$ gives rise to the obvious basis $B_{E,\ell,X}^{(m)}$ of the $\Q_{\ell}$-vector space
$\oplus_{i=1}^m V_{\ell}(X)=V_{\ell}(Y)$ that enjoys the following properties. If
$$u=\left(u_{ij}\right)_{i,j=1}^m \in \Mat_m(E)\subset\Mat_m(\End^0(X))=\End^0(Y)$$
then the matrix  of $u_{\ell} \in  \End_{\Q_{\ell}}(V_{\ell}(Y))$ w.r.t. $B_{E,\ell,X}^{(m)}$ coincides with the block matrix
$\left(M(u_{ij})\right)_{i,j=1}^m$. In particular, all its entries lie in $\Q$ and  do not depend on the choice of $\ell$.
\end{itemize}
\end{exs}

The paper is organized as follows. In Section \ref{plan} we discuss the plan of the proof of Theorem \ref{mat}. 
In particular, we obtain that Theorem \ref{mat} follows from  certain auxiliary assertions about
endomorphism subalgebras of abelian varieties (Theorems \ref{splitr0} and \ref{matrixE})
and about elements of finite-dimensional semisimple algebras in characteristic 0 (Theorem \ref{mat0}).
The results about endomorphism subalgebras are proven in 
Section \ref{aux}. The assertion about semisimple algebras is proven in 
Section \ref{LinAl}.

{\bf Acknowledgements}.  I am grateful to Fei Xu for interesting stimulating questions about 
 endomorphisms of abelian varieties and useful comments to the preliminary version of the paper. My special thanks go to the referee, whose thoughtful comments helped to improve the exposition.

\section{Endomorphism subalgebras of abelian varieties: statements}
\label{plan}
In the course of the proof of Theorem \ref{mat} we will use the following assertions.

\begin{thm}
\label{splitr0}
Let $D$ be a finite-dimensional  $\Q$-algebra with identity element $1_D$ and  let $W$ be a positive-dimensional
abelian variety over $K$ endowed with a $\Q$-algebra embedding
$\tau: D \hookrightarrow \End^0(W)$
that sends $1_D$ to $1_W$. 
Suppose that $r \ge 2$ is an integer and  $D$ splits into a direct sum
$D=\oplus_{i=1}^r  D_i$
of $r$ nonzero finite-dimensional  $\Q$-algebras $D_i$.
We will identify $D_i$'s with the corresponding two-sided ideals in $D$.
Then for all $i=1, \dots, r$ there exist positive-dimensional abelian subvarieties $W_i\subset W$
and $\Q$-algebra embeddings
$\tau_i: D_i \hookrightarrow \End^0(W_i)$
that send $1_{D_i}$ to $1_{W_i}$ and enjoy the following properties.

\begin{itemize}
\item[(i)]
The homomorphism of abelian varieties
$$S: \prod_{i=1}^r W_i \to W, \ \{w_i\}_{i=1}^r  \mapsto \sum_{i=1}^r w_i \
\text{ is an isogeny}.$$
\item[(ii)]
 For each  
$ u=\sum_{i=1}^r u_i \in D\ \text{ with  } u_i \in D_i \ \forall i$
we have
$$\{\tau_i(u_i)\}_{i=1}^r \in \oplus_{i=1}^r \End^0(W_i)\subset \End^0\left(\prod_{i=1}^r W_i\right),$$
$$S \circ \left(\{\tau_i(u_i)\}_{i=1}^r\right)\circ S^{-1}=\tau(u) \in \End^0(W).$$
\end{itemize}
\end{thm}

\begin{thm}
\label{matrixE}
Let $E$ be a number field, $m$ a positive integer, $\Mat_m(E)$ the matrix algebra of size $m$ over $E$, and $Z$ an abelian variety of positive dimension over $K$ endowed with a $\Q$-algebra embedding
$\tilde{\kappa}: \Mat_m(E) \hookrightarrow \End^0(Z)$
that sends the identity matrix $\II_m \in  \Mat_m(E)$ to $1_Z$.

Then  there are a positive-dimensional abelian variety $X$ over $K$, a $\Q$-algebra embedding
$\kappa: E \hookrightarrow \End^0(X)$
that sends $1$ to $1_X$, 
and an isogeny of abelian varieties $\psi: X^m \to Z$  that enjoy the following properties.

Let
$\Mat_m(\kappa): \Mat_m(E) \hookrightarrow \Mat_m(\End^0(X))=\End^0(X^m)$
be the natural $\Q$-algebra embedding
that sends a matrix $\left(a_{ij}\right)_{i,j=1}^m \in  \Mat_m(E) $ to
$$\left(\kappa(a_{ij})\right)_{i,j=1}^m \in  \Mat_m(\End^0(X))=\End^0(X^m).$$
Then
\begin{equation}
\label{isog}
\tilde{\kappa}(u)=\psi\circ \Mat_m(\kappa)(u)\circ \psi^{-1} \ \forall u\in \Mat_m(E).
\end{equation}
\end{thm}

\begin{rem}
\label{matrixEbasis}
Notice that in the notation of Theorem \ref{matrixE},  $\psi$ induces (by functoriality of Tate modules) the isomorphism 
$$\psi_{\ell}: V_{\ell}(X^m) \cong V_{\ell}(Z)$$
of $\Q_{\ell}$-vector spaces.   Thus $\psi_{\ell}\left(B_{E,\ell,X}^{(m)}\right)$ is a basis of the $\Q_{\ell}$-vector space
$ V_{\ell}(Z)$.
It follows from \eqref{isog} combined with Example \ref{isotype}(ii)
that for each $u \in  \Mat_m(E) $ there exists a square matrix $M(u)$ of size $2\dim(Z)$ with rational entries
that enjoys the following property. For each prime $\ell \ne \fchar(K)$ 
 the matrix of $\tilde{\kappa}(u)_{\ell}$ with respect to   $\psi_{\ell}\left(B_{E,\ell,X}^{(m)}\right)$ coincides with  $M(u)$.
\end{rem}

In light of  Theorem \ref{splitr0}, Theorem \ref{matrixE} and Remark \ref{matrixEbasis},
Theorem \ref{mat} is an immediate corollary of the following assertion
applied to semisimple $\End^0(X)$.

\begin{thm}
\label{mat0}
Let $D$ be a finite-dimensional semisimple algebra over $\Q$. Then every element of $D$ is contained
in a subalgebra of $D$ with the same $1$ that is isomorphic to a direct sum of matrix algebras over number fields.
\end{thm}

We prove Theorem \ref{mat0} in Section \ref{LinAl}. Theorems \ref{splitr0} and \ref{matrixE} will be proven in Section \ref{aux}.

\section{Endomorphism subalgebras of abelian varieties: proofs}
\label{aux}
Results of  this section (Lemma \ref{conjIsog} and    Theorem \ref{splitr}) and their proofs are rather straightforward (and boring). However, we need them in order to prove  Theorem \ref{splitr0}.

Throughout this section, $D$ be a finite-dimensional  $\Q$-algebra with identity element $1_D$ and  let $W$ be a positive-dimensional
abelian variety over $K$ endowed with a $\Q$-algebra embedding
$$\tau: D \hookrightarrow \End^0(W)$$
that sends $1_D$ to $1_W$. 

\begin{lem}
\label{conjIsog}
Let $u_1,u_2$ be two conjugate elements of $D$, i.e., there exists $s \in D^{*}$ such that
$$u_2=s u_1 s^{-1}.$$
Let $N$ be a positive integer such that all three elements
$$N\tau(u_1)=\tau(Nu_1), N\tau(u_2)=\tau(Nu_2), N\tau(s)=\tau(Ns) \in \End^0(W)$$
actually lie in $\End(W)$.
Let us consider abelian subvarieties 
$$W_1:=\tau(Nu_1)(W) \subset W, \ W_2:=\tau(Nu_2)(W)\subset W$$
of $W$. Then  
$$\tau(Ns) (W_1)=W_2.$$
 In  addition, the restriction
 $$\tau(Ns)\bigm|_{W_1}: W_1 \to W_2$$
 is an isogeny of abelian varieties.
\end{lem}

\begin{proof}
Renaming $Nu_1,Nu_2,Ns$ by  $u_1,u_2,s$ respectively, we may and will assume
that
$$\tau(u_1),\tau(u_2), \tau(s)\in \End(W), \ N=1.$$
Since $s$ in invertible in $D$, the endomorphism 
$$\tau(s): W \to W$$
is invertible in $\End^0(W)$ and therefore is an isogeny. This means that
\begin{equation}
\label{surS}
\tau(s)(W)=W,
\end{equation}
and there is a positive integer $r$ such that  multiplication by $r$ kills $\ker(\tau(s))$.
On the other hand, the equality $u_2=s u_1 s^{-1}$  means that
$u_2 s =s u_1$ and therefore
$$\tau(u_2)\tau(s)=\tau(s)\tau(u_1).$$
Combining it with \eqref{surS} and the definition of $W_1$ and $W_2$,
we obtain that
$$W_2=\tau(u_2)(W)=\tau(u_2)\tau(s)(W)=\tau(s)\tau(u_1)(W)=\tau(s)W_1,$$
i.e.,  $$W_2=\tau(s)W_1.$$ 
This means that the restriction
$$\tau(s)\bigm|_{W_1}: W_1 \to W_2,  \ w_1 \mapsto s(w_1)$$
 is a surjective morphism of abelian varieties.  Since its kernel
   $\ker(\tau(s)\bigm|_{W_1}$ is a group subscheme of $\ker(\tau(s))$, 
it is also killed by multiplication by $r$.
This implies that $\tau(s)\bigm|_{W_1}$ is an isogeny.
\end{proof}

The following assertion contains Theorem \ref{splitr0}.

\begin{thm}
\label{splitr}
Suppose that $r \ge 2$ is an integer and  $D$ splits into a direct sum
$$D=\oplus_{i=1}^r  D_i$$
of $r$ nonzero finite-dimensional  $\Q$-algebras $D_i$.
We will identify $D_i$'s with the corresponding two-sided ideals in $D$.
Let us consider the subrings
$$O_i:=\{u_i \in D_i\subset D\mid \tau(u_i)\in \End(W)\} $$
of $D_i$  ($1 \le i \le r$).

Let $e_i:=1_{D_i}$ be the identity element of $D_i$ viewed as a nonzero central idempotent in $D$.  Let $N$ be a positive integer such that
$Ne_i\in O_i$ for all $i$, i.e., all 
$\tau(Ne_i)=N\tau(e_1)$ lie in $\End(W)$.
(Such an $N$ always exists.)
Let us consider abelian subvarieties 
$W_i:=\tau(Ne_i)(W) \subset W$
of $W$. 

Then:
\begin{itemize}
\item[(i)]
\begin{enumerate}
\item[(a)]
The natural $\Q$-algebra homomorphisms
$$O_i\otimes\Q \to D_i, \ u_i\otimes c \mapsto c\cdot u_i$$
are isomorphisms for $i=1,\dots, r$.

\item[(b)] For each   
$$i,j \in \{1, \dots, r\}, \ i\ne j$$
and $w_i\in W_i, w_j\in W_j$
\begin{equation}
\label{Wij}
\tau(Ne_i)(w_i)=Nw_i, \tau(Ne_j)(w_i)=0.
\end{equation}
\end{enumerate}
\item[(ii)]
The natural morphisms of abelian varieties
\begin{equation}
\label{sumr}
S: \prod_{i=1}^r W_i \to W, \ \{w_i\}_{i=1}^r  \mapsto \sum_{i=1}^r w_i, \ P: W \to \prod_{i=1}^r W_i , \ w \mapsto 
\{\tau(Ne_i)(w_i)\}_{i=1}^r  
\end{equation}
are isogenies such that
\begin{equation}
\label{sum2r}
P\circ S=N_W, \  S \circ P=N_{\tilde{W}} \ \text{ where } \tilde{W}:=\prod_{i=1}^r W_i.
\end{equation}
In particular, $\sum_{i=1}^r W_i=\{\sum_{i=1}^r w_i\mid w_i\in W_i \ \forall \ i\}$ coincides with $W$.
\item[(iii)]
For each $u_i\in O_i$ and $u_j \in O_j$ with $i\ne j$
\begin{equation}
\label{uijWij}
\tau(u_i)(W_i)\subset W_i, \tau(u_j)(W_i)=\{0\}, \ \tau(u_i)(W_j)=\{0\}, \tau(u_j)(W_j)\subset W_j.
\end{equation}
\item[(iv)]
There   exist $\Q$-algebra embeddings 
$$\tau_i: D_i \to \End^0(W_i) $$
that  enjoy the following properties.
\begin{enumerate}
\item[(c)]
 $\tau_i(e_i)=1_{W_i}$.
\item[(d)] If $u_i \in O_i \subset D_i$ then $\tau_i(u_i) \in \End(W_i)$ and 
$$\tau_i(u_i)(w_i)=\tau(u_i)(w_i) \in W_i  \ \forall w_i \in W_i.$$
\item[(e)] For each  
$$ u=\sum_{i=1}^r u_i \in D\ \text{ with  } u_i \in D_i \ \forall i$$
we have
$$\{\tau_i(u_i)\}_{i=1}^r \in \oplus_{i=1}^r \End^0(W_i)\subset \End^0\left(\prod_{i=1}^r W_i\right)=\End^{0}(\tilde{W})$$
and
$$S \circ \left(\{\tau_i(u_i)\}_{i=1}^r\right) \circ S^{-1}=\tau(u) \in \End^0(W).$$
\end{enumerate}
\end{itemize}
\end{thm}

\begin{rem}
\begin{itemize}
\item[(i)]
The rings $O_i$ do not have to have the identity elements.
\item[(ii)]
The abelian subvarieties $W_i$ do not depend on the choice of $N$.
\end{itemize}
\end{rem}

\begin{proof}[Proof of Theorem \ref{splitr}]
(i). Since $\End(W)$ is a free $\Z$-module of finite rank and $\tau$ is a ring embedding,
$O_i$ is a also a free $\Z$-module of finite rank that generates the finite-dimensional $\Q$-vector space $D_i$.
This implies that $O_i\otimes\Q \to D_i$ is an isomorphism, which proves (i).

Clearly, 
$$\sum_{i=1}^r e_i=1_D, \  e_i^2=e_i \ \forall i,  \  e_i e_j=0 \ \forall i \ne j.$$
This implies that
\begin{equation}
\label{e1e2}
\tau(Ne_i)^2=N \cdot \tau(Ne_i),  \  \sum_{i=1}^r \tau(Ne_i)=N\cdot 1_W=N_W,
\end{equation}
$$ \tau(Ne_i) \tau(Ne_j) =0_W \ \forall i \ne j.$$

 Let 
 $$i,j \in \{1, \dots, r\}, \ w_i\in W_i, w_j \in W_j.$$
   Then there  exist
$$\tilde{w}_i,\tilde{w}_j \in W \text{ such that } w_i=\tau(Ne_i)(\tilde{w}_i) , w_j=\tau(Ne_j)(\tilde{w}_j).$$
This implies that
$$\tau(Ne_i) (w_i)=\tau(Ne_i)^2(\tilde{w}_i)=N\tau(Ne_i)(\tilde{w}_i)=N w_i.$$
On the other hand, if $i\ne j$  then
$$\tau(Ne_i) (w_j)=\tau(Ne_i)\tau(Ne_j)(\tilde{w}_j)=0\in W$$
in light of \eqref{e1e2}.

(ii) Recall that for each $w \in W$
$$\tau(Ne_i)(w) \in W_i \ \forall i \in \{1, \dots, r\}.$$
This implies that
$$ \ S\circ P(w)=S (\tau(Ne_1)(w), \tau(Ne_2)(w), \dots , \tau(Ne_r)(w))=\sum_{i=1}^r \tau(Ne_i)(w)=$$
$$\tau(N\cdot 1_D)(w)=\tau(N(\sum_{i=1}^r e_i))w=Nw.$$
This proves that 
\begin{equation}
\label{SPi}
 S\circ P=N_W.
 \end{equation}
On the other hand, let $w_i\in W_i  \ \forall i \in \{1, \dots, r\}$. 
Thanks to \eqref{e1e2},
$$P\circ S (w_1,w_2, \dots, w_r)=P\left(\sum_{i=1}^r w_i\right)=$$
$$ \quad  \quad  \quad \left(\sum_{i=1}^r\tau(Ne_1)(w_i), \sum_{i=1}^r\tau(Ne_2)(w_i), \dots, \sum_{i=1}^r \tau(Ne_r)(w_i)\right)=$$
$$(\tau(Nu_1)(w_1),\tau(Nu_2)(w_2), \dots, \tau(Nu_r)(w_r))   =$$ $$ (Nw_1,Nw_2, \dots, Nw_r)=N\cdot (w_1, w_2, \dots, w_r).$$
This proves that
\begin{equation}
\label{PiS}
P\circ S =N_{\tilde{W}}.
\end{equation}
Combining \eqref{SPi} and \eqref{PiS}, we obtain that both $S$ and $P$ are isogenies. This proves (ii).

Proof of (iii). 
Let $u_i \in O_i\subset D_i$. Then 
$$u_i=e_i u_i= e_i u_i e_i.$$
This implies that
$$\tau(Nu_i)=N \tau(u_i)=\tau(Ne_i)\tau(u_i) \in \End(W)$$
and therefore
$$\tau(Nu_i)(W)=\tau(Ne_i)\tau(u_i)(W)\subset \tau(Ne_i)(W)=W_i,$$
i.e.,
$$N\tau(u_i)(W)\subset W_i.$$
Clearly, $\tau(u_i)(W)$ is an abelian subvariety of $W$ and therefore coincides
with $N \tau(u_i)(W)$. This implies that
$$\tau(u_i)(W)=N \tau(u_i)(W)\subset W_i.$$
In particular, 
$$\tau(u_i)(W)\subset W_i \ \forall i.$$
On the other hand, if $i \ne j$  then
$$0= u_i e_j,  0_W=\tau(Nu_i e_j)=\tau(u_i)\tau(Ne_j)$$
and therefore
$$\tau(Nu_i)(W_j)=\{0\}.$$
Again, $\tau(u_i)(W_j)$ is an abelian subvariety of $W$ and therefore coincides
with 
$$N\tau(u_i)(W_j)= \tau(Nu_i)(W_j)=\{0\}.$$
Hence,
$\tau(u_i)(W_j)=\{0\}$.
This ends the proof of  \eqref{uijWij}.

(iv) Now \eqref{uijWij} allows us to define the ring homomorphisms
$$\tau_i: O_i \to \End(W_i),  \ u_i \mapsto \{w_i \mapsto \tau(u_i)(w_i)\in W_i \ \forall w_i \in W_i\} \ \forall u_i\in O_i .$$
It follows from already proven (i) that $\tau_i$ extend to $\Q$-algebra homomorphisms
$$D_i=O_i\otimes \Q \to \End(W_i)\otimes\Q=\End^0(W_i),$$
which we continue to denote by $\tau_i$. 
It follows from already proven (i) that $\tau_i(Ne_i)=N_{W_i}$ and therefore
$$\tau_i(e_i)=\frac{1}{N}\tau_i(Ne_i)=\frac{1}{N}N_{W_i}=1_{W_i}.$$
This proves (c). The property (d) follows from the very definition of $\tau_i$.
Let us prove that 
$$\tau_i: D_i \to \End^0(W_i)$$
is an {\sl embedding}. Suppose that $u_i \in D_i$ satisfies $\tau_i(u_i)=0$.
Replacing $u_i$ by $m u_i$ for sufficiently divisible positive integer $m$, we may and will
assume that $u_i \in O_i$. Then 
$$\{0\}=\tau_i(u_i)(W_i)=\tau(u_i)(W_i).$$  On the other hand,
if $i \ne j$  then $\tau(u_i)(W_j)=0$. It follows that
$\tau(u_i)$ kills 
$$W_i+\sum_{j\ne i}W_j=\sum_{j=1}^r W_j=W,$$ because the {\sl isogeny} $S$ is {\sl surjective}. This implies that $\tau(u_i)=0_W$. Since $\tau$ is injective,
$u_i=0$, i.e., $\tau_i$ is an embedding.

Let us prove (e).  Replacing $$u=(u_1, \dots, u_i, \dots, u_r)=\sum_{i=1}^r u_i\in \oplus_{i=1}^r D_i=D$$ by $mu=(mu_1,\dots mu_i, \dots, mu_r)$ for sufficiently divisible positive integer $m$,
we may and will assume that all $u_i\in O_i$. Let us check that in $\Hom(\prod_{i=1}^r W_i,W)=\Hom(\tilde{W},W)$ we have
\begin{equation}
\label{Stau}
S \circ (\tau_1(u_1),\dots \tau_i(u_i), \dots \tau_r(u_r)) =\tau(u)\circ S.
\end{equation}
So, let $(w_1,\dots w_i, \dots, w_r) \in \prod_{i=1}^r W_i$. Then
$$(\tau_1(u_1),\dots \tau_i(u_i), \dots, \tau_r(u_r))(w_1,\dots w_i, \dots, w_r)=\left(\tau_1(u_1)(w_1),\dots \tau_i(u_i)(w_i), \dots, \tau_r(u_r)(w_r)\right)\in
\prod_{i=1}^r W_i,$$
$$S\circ \left((\tau_1(u_1),\dots, \tau_i(u_i),  \dots, \tau_r(u_r))(w_1,\dots w_i, \dots w_r)\right)=S \left(\tau(u_1)(w_1),  \dots, \tau(u_i)(w_i), \dots, \tau(u_r)(w_r) \right)=$$
$$\sum_{i=1}^r\tau(u_i)(w_i).$$
On the other hand,
$S(w_1,\dots, w_i, \dots, w_r)=\sum_{i=1}^r w_i$; in addition,  thanks to   (iii),
 $$\tau(u_i)(w_j)=0 \ \forall i \ne j.$$
This implies that
$$ \tau(u)\circ S (w_1,\dots ,w_i, \dots,  w_r)=\tau(u)\left(\sum_{i=1}^r w_i\right)=\left(\sum_{i=1}^r\tau(u_i)\right)\left(\sum_{j=1}^r w_j\right)=$$
$$\sum_{i=1}^r \tau(u_i)(w_i)=
S\circ \left(\tau_1(u_1)(w_1),\dots, \tau_i(u_i)(w_i), \dots, \tau_r(u_r)(w_r)\right).$$
This proves \eqref{Stau}. Multiplying both sides of \eqref{Stau} by $ S^{-1}$ from the right, we get  the desired equality
$$S \circ (\tau_1(u_1),\dots \tau_i(u_i), \dots, \tau_r(u_r))\circ S^{-1}=\tau(u) $$
in $\End^0(W)$.
\end{proof}

\begin{proof}[Proof of Theorem \ref{matrixE}]
We write $e_{ij}\in \Mat_m(E)$ for the matrix, whose only nonzero entry is $1$ at the intersection of $i$th row and $j$th column,
 $i,j \in \{1, \dots , m\}$. We have
$$e_{ii}^2=e_{ii},  \sum_{i=1}^m e_{ii}=\II_m, \ e_{ii}e_{jj}=0 \ \forall i \ne j.$$
In addition, if $i \ne j$ then the monomial matrix 
$$s_{ij}=s_{ji}\in \GL(m,\Q) \subset \Mat_m(E)$$ attached to the {\sl transposition} $(ij)$ satisfies
$$s_{ij}=\II_m -(e_{ii}+e_{jj})+(e_{ij}+e_{ji}) \in \GL(m,\Q), \ s_{ij}^2=\II_m,$$
\begin{equation}
\label{Z1i}
   s_{ij} e_{ii} s_{ij}^{-1}=e_{jj}.
  \end{equation}
There is a positive integer $N$ such that
$$N\tilde{\kappa}(e_{ij}) \in \End(Z) \ \forall i,j=1, \dots m.$$
Let us consider the nonzero abelian subvarieties
$$Z_i:=\tilde{\kappa}(N e_{ii})(Z) \subset Z.$$
It follows from \eqref{Z1i} combined with Lemma \ref{conjIsog}
that there  are isogenies of abelian varieties
$$P_{ij}:=\tilde{\kappa}(N s_{ij})\bigm|_{Z_i}: Z_i \to Z_j, \  z_i \mapsto \tilde{\kappa}(N s_{ij})(z_i),$$
Since $$s_{ij}^2=\II_m, \  \tilde{\kappa}(s_{ij})^2=1_Z,$$ 
we get
$$P_{ji} \circ P_{ij}=N^2_{Z_i}.$$
They give rise to the product-isogenies
$$\mathbf{P}:=\prod_{i=1}^m P_{1i}: Z_1^m \to \prod_{i=1}^m Z_i, \ \mu:=\prod_{i=1}^m P_{i1}:  \prod_{i=1}^m Z_i \to Z_1^m$$
such that
$$\mu \circ \mathbf{P}=N^2_{Z_1^m}, \ \mathbf{P} \circ \mu =N^2_{\tilde{Z}} \ \text{ where } \tilde{Z}:=\prod_{i=1}^m Z_i.$$
Applying Theorem \ref{splitr} to the subalgebra of diagonal matrices
$$D=\oplus_{i=1}^m E\cdot e_{ii}\subset \Mat_m(E),$$
$$r=m, D_i=E\cdot e_{ii}, \  e_i:=e_{ii}$$
and 
$$W:=Z,  \ \tau: D\hookrightarrow \End^0(Z), \ u \mapsto \tilde{\kappa}(u),$$
we obtain that the morphism of abelian varieties
\begin{equation}
\label{sumZ}
S: \tilde{Z}=\prod_{i=1}^m Z_i \to Z, \  \{z_i\}_{i=1}^m  \mapsto \sum_{i=1}^r z_i 
\end{equation}
is an  isogeny.
In addition, we get the $\Q$-algebra embeddings
$$\tau_i:   E\cdot e_{ii}=D_i \hookrightarrow \End^0(Z_i)$$
that send $e_{ii}$ to $1_{Z_i}$ and such that  for every collection $u_i \in D_i$ ($1\le i\le m$)
$$\sum_{i=1}^m \tau_i(u_i) \in \oplus_{i=1}^m \End^0(Z_i)\subset \End^0\left(\prod_{i=1}^m Z_i\right)=\End^{0}(\tilde{Z}),$$
$$S\circ \left(\sum_{i=1}^m \tau_i(u_i)\right)\circ S^{-1}=\tilde{\kappa}\left(\sum_{i=1}^m u_i\right) \in \End^0(Z).$$
This implies that
$$S^{-1}\circ \tilde{\kappa}\left(\sum_{i=1}^m u_i\right)\circ S=\sum_{i=1}^m \tau_i(u_i) \in \oplus_{i=1}^m \End^0(Z_i)\subset \End^0\left(\prod_{i=1}^m Z_i\right).$$
In particular, for each $j\in \{1, \dots,m\}$
$$S^{-1}\circ \tilde{\kappa}(e_{jj})\circ S =\tau_j(e_{jj})=1_{Z_j}\in \End^0(Z_j)\subset \oplus_{i=1}^m \End^0(Z_i) \subset \End^0\left(\prod_{i=1}^m Z_i\right).$$

Let us put
$$X:=Z_1, \psi:=  \mathbf{P} \circ S: X^m=Z_1^m \to \prod_{i=1}^m Z_i \to Z.$$
In order to define $\kappa$, let us consider
$$O:=\{u \in E\mid \tilde{\kappa}(u)\in \End(Z)\}.$$
Clearly, $O$ is an order in $E$ and the natural $\Q$-algebra homomorphism
$$O\otimes\Q \to E, \ u\otimes a \mapsto au$$
is an isomorphism.  In addition, for each nonzero $u \in O$ the selfmap of $Z$
$$\tilde{\kappa}(u): Z \to Z$$
is an isogeny.
Since $E$ is the center of $\Mat_m(E)$, 
every $u\in O$ commutes with all $e_{ij}$ and $s_{ij}$. In particular, for all $u\in O$
$$\tilde{\kappa}(u)(Z_i)=\tilde{\kappa}(u)\tilde{\kappa}(Ne_{ii})(Z)=\tilde{\kappa}(uNe_{ii})(Z)=\tilde{\kappa}(Ne_{ii}u)(Z)\subset \tilde{\kappa}(Ne_{ii})(Z)=Z_i$$
and the diagrams
\begin{equation}
\begin{CD}
Z_1 @>{\tau_1(u)}>>Z_1\\
@V P_{1i} VV @V P_{1i} VV \\
Z_i @>{\tau_i(u)}>> Z_i
\end{CD},\ \quad\quad
\begin{CD}
Z_i @>{\tau_i(u)}>>Z_i\\
@V P_{i1} VV @V P_{i1} VV \\
Z_1 @>{\tau_1(u)}>> Z_1
\end{CD}
\end{equation}
are commutative. This gives rise to the ring homomorphisms
$$\tau_i: O \to \End(Z_i),  \ u \mapsto \{z_i \mapsto \tilde{\kappa}(u)(z_i) \ \forall z_i \in Z_i\} \ \forall u \in O,$$
which are obviously injective, and extend them by $\Q$-linearity to the injective $\Q$-algebra homomorphisms
$E \to \End^{0}(Z_i)$, which we continue to denote by $\tau_i$. Let us put
$$\kappa:=\tau_1: E \hookrightarrow \End^0(X=Z_1).$$

\end{proof}

\section{Linear Algebra}
\label{LinAl}

\begin{defn}
Let $k$ be a field and $\mathcal{C}$ be a finite-dimensional $k$-algebra with identity element $1_{\mathcal{C}}$.
We say that  $\mathcal{C}$ is splittable over $k$ if it  is isomorphic to a direct sum of matrix algebras over fields that are finite algebraic separable extension of $k$.
\end{defn}

\begin{exs}
\label{EXS}
\begin{itemize}
\item[(i)]
If $\mathcal{C}$  is a field that is  a finite algebraic separable extension of $k$ then it is splittable over $k$.
\item[(ii)]
If $\mathcal{C}_1$ and $\mathcal{C}_2$ are splittable over $k$ then their direct sum  $\mathcal{C}_1 \oplus\mathcal{C}_2$ is also splittable over $k$.
\item[(iii)]
If $\mathcal{C}$ is  splittable over $k$ then for all positive integers $d$ the matrix algebra $\Mat_d(\mathcal{C})$ is also  splittable over $k$.
\item[(iv)]
Let $L/k$ be a finite algebraic separable field extension of $k$. If $\mathcal{C}$ is  splittable over $k$ then
$$\mathcal{C}_L=\mathcal{C}\otimes_k L$$
is splittable over $k$ and over $L$.
\item[(v)] Suppose that $E$ is an overfield of $k$ and $\mathcal{C}$ is   a finite-dimensional splittable $E$-algebra. 
If $E/k$ is a finite separable field extension then $\mathcal{C}$ is  
 splittable over $k$ as well.
\end{itemize}
\end{exs}

Clearly, Theorem \ref{mat0} is a special case of the following assertion.

\begin{thm}
\label{elementsplit}
Let $k$ be a field of characteristic zero, $\mathcal{A}$ a nonzero semisimple finite-dimensional $k$-algebra with identity element $1_{\mathcal{A}}$, and $f$ an element  of $\mathcal{A}$. Then there exists a  $k$-subalgebra $\mathcal{C}$ of $\mathcal{A}$ that contains $1_{\mathcal{A}}$ and $f$, and is splittable over $k$.
\end{thm}

We will need the following  lemma that will be proven at the end
of this section.

\begin{lem}
\label{dense}
Let $\mathcal{A}$  be a nonzero finite-dimensional algebra over a field $k$ with identity element $1_{\mathcal{A}}$.
Let $\mathcal{V}$ be a nonzero finite-dimensional $k$-vector space that is a faithful semisimple $\mathcal{A}$-module.  Assume additionally
that every simple $\mathcal{A}$-submodule $M$ of $\mathcal{V}$ is absolutely simple, i.e., the centralizer
$\End_{\mathcal{A}}(M)$ of $\mathcal{A}$ in $\End_k(M)$ coincides with $k\cdot 1_M$ where $1_M: M \to M$
is the identity map.

Then $\mathcal{A}$ is isomorphic to a direct sum of matrix algebras over $k$.

\end{lem}

\begin{proof}[Proof of Theorem \ref{elementsplit}]
Induction by $\dim_k(\mathcal{A})$. 

{\bf Step 0}. If $\dim_k(\mathcal{A})=1$ then $\mathcal{A}=k\cdot 1_{\mathcal{A}}\cong k$ is obviously splittable over $k$.

Now assume that $d:=\dim_k(\mathcal{A})>1$ and the assertion of Theorem \ref{elementsplit} holds true
for all semisimple algebras of dimension $<d$ over any field of characteristic $0$.

{\bf Step 1}. Suppose that the $k$-algebra $\mathcal{A}$ is not simple, i.e., it splits into a direct sum
$$\mathcal{A}=\mathcal{A}_1\oplus \mathcal{A}_2$$
 of two nonzero semisimple $k$-algebras $\mathcal{A}_1$ and $\mathcal{A}_2$.  Clearly, both $\dim_k(\mathcal{A}_1)$
 and $\dim_k(\mathcal{A}_2)$ are strictly less than $d$. There are elements 
 $$f_1\in \mathcal{A}_1, \  f_2\in \mathcal{A}_2$$
 such that $f=f_1+f_2$. Applying the induction assumption to both $(\mathcal{A}_1,f_1)$ and $(\mathcal{A}_2,f_2)$,
 we obtain that there are splittable over $k$ subalgebras 
$$\mathcal{C}_1\subset \mathcal{A}_1, \ \mathcal{C}_2\subset \mathcal{A}_2$$
such that
$$1_{\mathcal{A}_1}, f_1\in \mathcal{C}_1, \ 1_{\mathcal{A}_2}, f_2\in \mathcal{C}_2.$$
Now the direct sum 
$$\mathcal{C}=\mathcal{C}_1\oplus \mathcal{C}_2 \subset \mathcal{A}_1\oplus \mathcal{A}_2=\mathcal{A}$$
is splittable over $k$ and contains
$1_{\mathcal{A}_1}+ 1_{\mathcal{A}_2}=1_{\mathcal{A}}$ and
$f_1+f_2=f$. 

So, in the course of the proof, we may and will assume that  $\mathcal{A}$ is a simple $k$-algebra.

{\bf Step 2}. Let $E$ be the center of the simple $k$-algebra $\mathcal{A}$. Then $E$ is a field
that is a finite algebraic extension of $k$. Clearly, 
$\mathcal{A}$ carries the natural structure of central simple $E$-algebra. If $E \ne k$ then
$$\dim_E(\mathcal{A})<\dim_k(\mathcal{A})$$
and the induction assumption implies that there is a splittable over $E$ subalgebra 
$\mathcal{C}\subset \mathcal{A}$ that contains both  $1_{\mathcal{A}}$ and $f$.
Since $E/k$ is finite algebraic, $\mathcal{C}$ is splittable over $k$ as well.

So, in the course of the proof, we may and will assume that $E=k$, i.e.,
$\mathcal{A}$ is a central simple $k$-algebra.

{\bf Step 3}.
So, $\mathcal{A}$ is a {\sl central simple} $k$-algebra of finite $k$-dimension $d>1$. Recall that $d=m^2$  where $m>1$ is an integer.
 Then $\mathcal{A}$ carries the natural structure
of a $m^2$-dimensional 
$k$-Lie algebra with brackets
$$[u,v]:=uv-vu \ \forall u,v \in \mathcal{A}.$$
The center of the $k$-Lie algebra $\mathcal{A}$ coincides with $k\cdot 1_{\mathcal{A}}$.
Let us consider the $k$-linear {\sl reduced trace map} (see \cite{Reiner})
$$\tr_{\mathcal{A}}: \mathcal{A} \to k.$$
Recall \cite{Reiner} that
$$\tr_{\mathcal{A}}(uv)=\tr_{\mathcal{A}}(vu) \ \forall u,v\in \mathcal{A}; \tr_{\mathcal{A}}(\alpha)=m\alpha \ \forall \alpha \in k=k\cdot 1_{\mathcal{A}}.$$
This implies that 
$$\tr_{\mathcal{A}}([u,v])=\tr_{\mathcal{A}}(uv)-\tr_{\mathcal{A}}(vu)=0,$$
and therefore one may view $\tr_{\mathcal{A}}$ as a homomorphism of $k$-Lie algebras (here $k$ is viewed as the Lie algebra with zero brackets operation).
It follows that
the $k$-Lie algebra $\mathcal{A}$ splits into a direct sum
$$\mathcal{A}=k\cdot 1_{\mathcal{A}} \oplus \sL(\mathcal{A})$$
of its center $k\cdot 1_{\mathcal{A}}$ and the nonzero $(m^2-1)$-dimensional
$k$-Lie algebra 
$\sL(\mathcal{A}):=\ker(\tr_{\mathcal{A}})$; in addition,  $\sL(\mathcal{A})$ contains the derived $k$-Lie subalgebra 
$[\mathcal{A},\mathcal{A}]$ of $\mathcal{A}$.
It is known that $[\mathcal{A},\mathcal{A}]$ is an absolutely simple $k$-Lie algebra of type ${\sf A}_{m-1}$ over $k$,
see \cite[Ch. X, Sect. 3]{Jacobson}. This implies that
$$\dim_{k}([\mathcal{A},\mathcal{A}])=m^2-1= \dim_k(\sL(\mathcal{A})).$$
Since $[\mathcal{A},\mathcal{A}]\subset \sL(\mathcal{A})$, we have
$[\mathcal{A},\mathcal{A}]= \sL(\mathcal{A}),$ 
and therefore
$$ \mathcal{A}=k\cdot 1_{\mathcal{A}} \oplus \sL(\mathcal{A})=k\cdot 1_{\mathcal{A}} \oplus [\mathcal{A},\mathcal{A}]$$
is a reductive $k$-Lie algebra, whose  center is $k\cdot 1_{\mathcal{A}}$ and the semisimple part  is $\sL(\mathcal{A})$.

Since $\fchar(k)=0$,   
$$\tr_{\mathcal{A}}(1)=m \ne 0 \ \text{ in } k.$$
Replacing $f$ by $f -\frac{\tr_{\mathcal{A}}(f)}{m}1_{\mathcal{A}}$, we may and will assume that
$f$ lies in (semi)simple $\sL(\mathcal{A}).$

{\bf Step 4}.
Since $f$ is an element of the semisimple $k$-Lie algebra $\sL(\mathcal{A})$, it can be presented as a sum
$$f=f_{s}+f_n$$
of commuting semisimple $f_{s}$ and nilpotent $f_n$ \cite[Sect. 6, n 3, Th. 3]{BourbakiI}. Let us consider the natural faithful representation
of $\sL(\mathcal{A})$ in the finite-dimensional $k$-vector space $\mathcal{A}$
$$\sL(\mathcal{A}) \to \End_k(\mathcal{A}),  u \mapsto \{z \mapsto uz \ \forall z \in \mathcal{A}\} \ \forall u \in \sL(\mathcal{A}).$$
Then $f_{s}  \in \sL(\mathcal{A})$ acts on $\mathcal{A}$ as a semisimple operator with zero trace and $f_{n} \in \sL(\mathcal{A})$ acts  on $\mathcal{A}$ as a nilpotent operator.

{\bf Step 5}. Suppose that $f_{s} \ne 0$. Then the subalgebra $k[f_s]$ of $\mathcal{A}$ generated by $f_s$ and $1_{\mathcal{A}}$
is a semisimple commutative $k$-subalgebra that contains $k\cdot 1_{\mathcal{A}}$ as a proper $k$-subalgebra. The centralizer
$\mathcal{Z}_{f_{s}}$ of semisimple  $k[f_s]$ in semisimple $\mathcal{A}$  is a {\sl semisimple} $k$-subalgebra (\cite[Th. 4.3.2]{Herstein}, \cite[Th. 4.1]{Zarhin}) that contains 
both
$f_s$ and $f_n$ and therefore
contains $f_s+f_n=f$. Clearly, $\mathcal{Z}_{f_{s}}$ does {\sl not} coincide with the whole $\mathcal{A}$, since $f_s$ does not belong to the center 
 $k\cdot 1_{\mathcal{A}}$. This implies that 
 $$\dim_k(\mathcal{Z}_{f_{s}})<\dim_k(\mathcal{A})=d.$$
 The induction assumption applied to $\mathcal{Z}_{f_{s}}$ proves the desired assertion.
 So, we may and will assume that $f_s=0$, i.e., $f=f_n$ is  a nonzero {\sl nilpotent} element.

 {\bf Step 6}. So, $f$ is a nonzero nilpotent element.
By Jacobson-Morozov theorem \cite[Ch. 8, Sect. 11, Prop. 2]{Bourbaki},  there is a Lie $k$-subalgebra $\mathfrak{g}$ of $\sL(\mathcal{A})$
that contains $f$ and is isomorphic to $\sL(2,k)$. Let $\mathcal{C}\subset \mathcal{A}$
be the associative $k$-subalgebra of $\mathcal{A}$ generated by $\mathfrak{g}$ and $1_{\mathcal{A}}$.
Clearly,
$$f \in \mathfrak{g}\subset \mathcal{C}\subset \mathcal{A}.$$
Let us consider the natural faithful action of $\mathcal{C}$ on the $k$-vector space
$\mathcal{V}:=\mathcal{A}$ induced by multiplication in $\mathcal{A}$. Clearly, a $k$-vector subspace $\mathcal{W}$ of $\mathcal{V}$
is a $\mathcal{C}$-submodule (resp. a simple $\mathcal{C}$-submodule) if and only it is a 
$\mathfrak{g}$-submodule (resp. a simple $\mathfrak{g}$-submodule). In addition, if $\mathcal{W}$ is a $\mathcal{C}$-submodule
then the centralizer
$$\End_{\mathcal{C}}(\mathcal{W})=\End_{\mathfrak{g}}(\mathcal{W}).$$
Recall \cite[Ch. VII, Sect. 1, Prop. 3 and Th. 1]{Bourbaki} that every  $\sL(2,k)$-module of finite $k$-dimension is semisimple; in addition,
if $\mathcal{W}$ is a simple $\sL(2,k)$-module of finite $k$-dimension then it is absolutely simple, ibid.  This implies that
$\mathcal{V}$ is a semisimple $\mathcal{C}$-module and all of its  simple submodules are absolutely simple.
It follows from Lemma \ref{dense} that $\mathcal{C}$ is a direct sum of matrix algebras over $k$ and therefore
is splittable. This ends the proof of Theorem  \ref{elementsplit}.
\end{proof}

\begin{proof}[Proof of Lemma \ref{dense}]
We may and will identify $\mathcal{A}$ with its isomorphic image in $\End_k(\mathcal{V})$.
The semisimplicity and faithfullness of $\mathcal{V}$ combined with  \cite[Ch. XVII, Sect. 4, Prop. 4.7]{Lang} 
imply that $\mathcal{A}$  is a finite-dimensional semisimple $k$-algebra.  Let us consider the centralizer
  $$\mathcal{B}:=\End_{\mathcal{A}}(\mathcal{V})\subset \End_k(\mathcal{V})$$
  of $\mathcal{A}$ in $\End_k(\mathcal{V})$.  The semisimplicity of $\mathcal{V}$ and absolute simplicity of all  its simple $\mathcal{A}$-submodules
  combined with  \cite[Ch. XVII, Sect. 1, Prop. 1.2]{Lang} imply that there are a positive integer $s$ and $s$ positive integers
  $d_1, \dots, d_s$ such that the $k$-algebra
  $$\mathcal{B}\cong \oplus_{i=1}^s \Mat_{d_i}(k).$$
  In particular, $\mathcal{B}$ is semisimple and therefore the faithful $\mathcal{B}$-module $\mathcal{V}$ is semisimple.
  Let
  $$p_i: \oplus_{i=1}^s \Mat_{d_i}(k)\twoheadrightarrow \Mat_{d_i}(k)$$
  be the (surjective) projection map to the $i$th summand. Recall that the coordinate $k$-vector space $k^{d_i}$ endowed with the natural
  action of $\Mat_{d_i}(k)$ is the only (up to an isomorphism) simple $\Mat_{d_i}(k)$-module and this module is absolutely simple.
  Using $p_i$, one may endow $k^{d_i}$  with the natural structure of $\oplus_{i=1}^s \Mat_{d_i}(k)$-module and this module, which we denote
  by $M_i$, is an absolutely simple $\oplus_{i=1}^s \Mat_{d_i}(k)$-module. Clearly, every simple $\oplus_{i=1}^s \Mat_{d_i}(k)$-module is
  isomorphic to one of $M_i$; in particular, each simple $\oplus_{i=1}^s \Mat_{d_i}(k)$-module is absolutely simple. This means that each 
  simple $\mathcal{B}$-module is absolutely simple. 
  Recall that  the  $\mathcal{B}$-module $\mathcal{V}$ is semisimple and therefore is isomorphic to a direct sum of simple 
   $\mathcal{B}$-modules, each of which is absolutely simple. It follows from \cite[Ch. XVII, Sect. 1, Prop. 1.2]{Lang} that the centralizer $\End_{\mathcal{B}}(\mathcal{V})$
   of $\mathcal{B}$ in $\End_k(\mathcal{V})$ is isomorphic to a direct sum of matrix algebras over $k$.
     By Jacobson's density theorem \cite[Ch. XVII, Sect. 3, Th. 3.2]{Lang}, our $\mathcal{A}$ coincides with $\End_{\mathcal{B}}(\mathcal{V})$
     and therefore is isomorphic to a direct sum of matrix algebras over $k$. This ends the proof.

\end{proof}


\begin{thebibliography}{MMMMM}
\bibitem{BourbakiI} N. Bourbaki, Groupes et Alg\`ebres de Lie, Chapitre I. Hermann, Paris, 1971.

\bibitem{Bourbaki} N. Bourbaki, Groupes et Alg\`ebres de Lie, Chapitres VII-VIII. Hermann, Paris, 1975.

\bibitem{Herstein} I.N.  Herstein, Noncommutative Rings. The Mathematical Association of America/John Wiley and Sons, 1968.

\bibitem{Jacobson} N. Jacobson, Lie Algebras. Dover Publications, Inc., New York, 1979.

\bibitem{Lang} S. Lang, Algebra, 2nd edition, Addison-Wesley, 1993,

\bibitem{Mumford} D. Mumford, Abelian varieties, 2nd edition. Oxford University Press, 1974.

\bibitem{Reiner} I. Reiner, Maximal orders.  Oxford University Press, 2003.

\bibitem{Ribet} K. Ribet, {\sl Galois action on division points of Abelian varieties with real multiplication}. Amer. J. Math. {\bf 98} (1976), 751--804.

\bibitem{Tate} J. Tate, {\sl Endomorphisms of Abelian varieties over finite fields}. Invent. Math. {\bf 2} (1966), 134--144.
\bibitem{Zarhin} Yu.G. Zarhin, {\sl Endomorphism Algebras of Abelian Varieties with Special Reference to Superelliptic Jacobians}. In: Geometry, Algebra, Number Theory,
and Their Information Technology Applications. Springer Proceedings
in Mathematics and Statistics {\bf 251} (2018), 477--528.



\end{thebibliography}
\end{document}